\documentclass[11pt, reqno]{amsart}
\usepackage{enumitem}
\makeatletter
\newcommand{\mylabel}[2]{#2\def\@currentlabel{#2}\label{#1}}
\makeatother
\usepackage{array}
\usepackage[english]{babel}
\usepackage[top=1in, bottom=1in, left=1in, right=1in]{geometry}
\newcolumntype{P}[1]{>{\centering\arraybackslash}p{#1}}

\expandafter\let\csname ver@amsthm.sty\endcsname\relax

\usepackage{tikz}

\usepackage{amsmath}
\usepackage{amssymb}
\usepackage{mathdots}
\usepackage{mathtools}

\usepackage{dsfont}
\usepackage{chngpage}

\usepackage{hyperref}
\usepackage{amsthm}
\usepackage[capitalize,noabbrev]{cleveref}

\allowdisplaybreaks

\numberwithin{equation}{section}

\newtheorem{thm}{Theorem}[section]
\newtheorem{lemma}[thm]{Lemma}

\newtheorem{con}[thm]{Conjecture}

\newtheorem{Example}[thm]{Example}

\newtheorem{Remark}[thm]{Remark}

\crefname{thm}{Theorem}{Theorems}
\crefname{lemma}{Lemma}{Lemmas}
\crefname{cor}{Corollary}{Corollaries}
\crefname{prop}{Proposition}{Propositions}

\crefname{example}{Example}{Examples}
\crefname{remark}{Remark}{Remarks}

\newcommand{\emailhref}[1]{\email{\href{#1}{#1}}}

\newcommand{\Zgtwo}{{\mathbb{Z}}_{\geq 2}}

\title[Generalized Glasby-Paseman sequences]{Unimodality and log-concavity of Generalized Glasby-Paseman sequences}

\author[Seok Hyun Byun]{Seok Hyun Byun}\emailhref{sbyun@amherst.edu}
\address{Department of Mathematics, Amherst College, Amherst, MA, U.S.A.}
\thanks{S.H.B. was supported in part by an AMS-Simons Travel Grant.}

\author[Svetlana Poznanovi\'c]{Svetlana Poznanovi\'c}\emailhref{spoznan@clemson.edu}
\address{School of Mathematical and Statistical Sciences, Clemson University, Clemson, SC, U.S.A.}
\thanks{S.P. was supported by Simons Foundation gift MP-TSM-00002798.}

\begin{document}

\begin{abstract}
In this paper, we consider a two-parameter ($l$ and $a$) generalization of a sequence that Glasby and Paseman considered. Based on computer experiments, we conjecture its unimodality, log-concavity, peak positions, and the asymptotic behavior of the maximum values. Then we prove this conjecture for the case where $l=2$ and $a=1$. We finish the paper by making some comments about the conjecture on the generalized sequence.
\end{abstract}

\maketitle

\section{Introduction}  \label{sec:Intro}

A sequence of real numbers $(x_{0},\ldots,x_{d})$ is \textit{unimodal} if it satisfies $x_{0}\leq\cdots\leq x_{j-1}\leq x_{j}\geq x_{j+1}\geq\cdots\geq x_{d}$ for some integer $j$. Many sequences in mathematics are unimodal, and proving their unimodality often requires new ideas, as there is no uniform way to prove the unimodality of a sequence. A sequence of nonnegative real numbers $\{y_{n}\}_{n\geq0}$ is \textit{log-concave} if $y_{i}^{2}\geq y_{i-1}y_{i+1}$ for every positive integer $i$. It is known that if a sequence is log-concave, it is unimodal. However, in general, establishing the log-concavity of sequences is also not straightforward. For readers who are interested in a general overview of these topics, we refer to \cite{MR3409348},\cite{brenti1989unimodal},\cite{brenti1994update}, and \cite{stanley1989log}.

In \cite{glasby2022maximum}, Glasby and Paseman considered the following interesting sequence:
\begin{equation}\label{GP1}
    \bigg\{\frac{1}{2^r}\sum_{i=0}^{r}\binom{m}{i}\bigg\}_{0\leq r\leq m}
\end{equation}
for nonnegative integers $m$. They showed that this sequence is unimodal and attains its maximum value at $r=\big\lfloor\frac{m}{3}\big\rfloor+1$ for all nonnegative integers $m$, except $m=0,3,6,9, \text{ and } 12$. Later, the authors \cite{byun2024maximum} considered a one-parameter generalization of this sequence:
\begin{equation}\label{BP1}
    \bigg\{\frac{1}{(1+a)^r}\sum_{i=0}^{r}\binom{m}{i}a^i\bigg\}_{0\leq r\leq m}
\end{equation}
for integers $m\geq2$ and $a\geq1$ and showed that this sequence is log-concave (and thus unimodal) and attains its maximum value at $r=\big\lfloor\frac{am+a+2}{2a+1}\big\rfloor$ for all $m$, except $m= 3, 2a+4, 4a+5$ (when $a=1$, we also exclude $m=6a+6=12$). See also \cite{glasby2024maximizing} for a recent work of Glasby and Paseman that generalized the sequence \eqref{GP1} in a different direction.

In this paper, we consider a two-parameter generalization of \eqref{GP1}, whose one-parameter deformation corresponds to \eqref{BP1}. Observe that the entry of the sequence \eqref{BP1} can be rewritten as follows:
\begin{equation}\label{BP2}
    \frac{1}{(1+a)^r}\sum_{i=0}^{r}\binom{m}{i}a^i=\frac{\sum_{i=0}^{r}\binom{m}{i}a^{i}}{\sum_{i=0}^{r}\binom{r}{i}a^{i}}.
\end{equation}
This equivalent expression allows one to easily generalize the sequence, as the numerator and denominator of \eqref{BP2} look very similar: they are both polynomials in $a$ with the same degree, and their coefficients exhibit similar patterns. Recall that for nonnegative integers $l$, $p_{l}(x_0,\ldots,x_r)\coloneqq \sum_{i=0}^{r}x_{i}^l$ is the $l$-th power sum symmetric polynomial. Motivated by the expression \eqref{BP2}, one can consider the following sequence
\begin{equation}\label{BP3}
    \bigg\{\frac{\sum_{i=0}^{r}\{\binom{m}{i}a^{i}\}^{l}}{\sum_{i=0}^{r}\{\binom{r}{i}a^{i}\}^{l}}\bigg\}_{0\leq r\leq m}=\bigg\{\frac{p_{l}(\binom{m}{0}a^{0},\ldots,\binom{m}{r}a^{r})}{p_{l}(\binom{r}{0}a^{0},\ldots,\binom{r}{r}a^{r})}\bigg\}_{0\leq r\leq m}
\end{equation}
for positive integers $m$ and $l$, and a positive real number $a$. For this expanded family of sequences, one can still ask the same questions: When is this sequence log-concave or unimodal? If the sequence is unimodal, what is the location of the maximum? What can we say about the limiting behavior of the maximum value? Based on computer experiments, we make the following conjecture. 

\begin{con}\label{Conjecture}
The sequence \eqref{BP3} satisfies the following properties.
\begin{enumerate}
    \item[(a)] The sequence \eqref{BP3} is unimodal for all $m,l\in \mathbb{Z}^{+}$ and $a\in\mathbb{R}^{+}$.
    \item[(b)] For $l\in \mathbb{Z}^{+}$ and $a\in\mathbb{R}^{+}$, the sequence \eqref{BP3} is log-concave for all but finitely many $m\in \mathbb{Z}^{+}$.
    \item[(c)] For $l,a\in \mathbb{Z}^{+}$ and $m\in \Zgtwo$, the sequence \eqref{BP3} attains its unique maximum value at an index $r$ that is equal to one of these three consecutive values: $\big\lfloor\frac{am+a+2}{2a+1}\big\rfloor-1$, $\big\lfloor\frac{am+a+2}{2a+1}\big\rfloor$, $\big\lfloor\frac{am+a+2}{2a+1}\big\rfloor+1$. 
    \item[(d)] For $l\in \mathbb{Z}^{+}$ and $a\in\mathbb{R}^{+}$, the maximum value of the sequence \eqref{BP3} is asymptotically\footnote{It means that the ratio between the maximum value of the sequence \eqref{BP3} and the expression given in Conjecture \ref{Conjecture} (d) approaches $1$ as $m$ approaches $\infty$.}
    \begin{equation*}
        \frac{l^{\frac{1}{2}}}{\sqrt{2\pi m}}\cdot\frac{(1+2a)^{\frac{1}{2}}(1+a)a^{\frac{l-2}{2}}}{(1+a)^l-1}\cdot\Big(\frac{1+2a}{1+a}\Big)^{(m+\frac{1}{2})l}.
    \end{equation*}
\end{enumerate}
\end{con}

The results in~\cite{glasby2022maximum} and ~\cite{byun2024maximum} confirm the conjecture for  $l=1$, $a=1$ and $l=1$, $a \in \mathbb{Z}^{+}$, respectively. In this paper, we confirm the conjecture for $l=2$, $a=1$. Note that in this case, using a combinatorial identity $\sum_{i=0}^{r}\binom{r}{i}^2=\sum_{i=0}^{r}\binom{r}{i}\binom{r}{r-i}=\binom{2r}{r}$, the sequence \eqref{BP3} can be rewritten as follows:
\begin{equation}\label{BP4}
    \bigg\{\frac{\sum_{i=0}^{r}\binom{m}{i}^{2}}{\sum_{i=0}^{r}\binom{r}{i}^{2}}\bigg\}_{0\leq r\leq m}=\bigg\{\frac{1}{\binom{2r}{r}}\sum_{i=0}^{r}\binom{m}{i}^{2}\bigg\}_{0\leq r\leq m}\eqqcolon\big\{g_{m}(r)\big\}_{0\leq r\leq m}.
\end{equation}
The main theorem of this paper is the following\footnote{The cases \(m=0\) and \(m=1\) yield the trivial sequences \((1)\) and \((1,1)\), respectively, so we restrict attention to \(m \ge 2\).}.
\begin{thm}\label{MainThm}
    For an integer $m\geq2$, 
    \begin{enumerate}
        \item[(a)] The sequence $\big\{g_{m}(r)\big\}_{0\leq r\leq m}$ is log-concave and thus unimodal.
        \item[(b)] This sequence attains its unique maximum value at $r=r_{m}\coloneqq\big\lfloor\frac{m+2}{3}\big\rfloor$.
        \item[(c)] The maximum value $g_{m}(r_{m})$ satisfies
        \begin{equation}
        \lim_{m\rightarrow\infty}g_{m}(r_{m})\cdot\frac{2^{2m}\sqrt{\pi m}}{3^{2m+\frac{1}{2}}}=1.
        \end{equation}
    \end{enumerate}
\end{thm}

While our proof of the above theorem follows the idea used in \cite{glasby2022maximum} and \cite{byun2024maximum}, the proof is more involved due to the complexity of the entries (their numerators involve sums of squares of binomial coefficients, rather than simple sums as in \cite{glasby2022maximum} and \cite{byun2024maximum}). Also, it turns out that the simplification of the denominator, $\sum_{i=0}^{r}\binom{r}{i}^2=\binom{2r}{r}$, is crucial in our proof of Theorem \ref{MainThm}. In particular, if 1) $l=2$ and $a\neq1$ or 2) $l\geq3$, the expression $\sum_{i=0}^{r}\{\binom{r}{i}a^i\}^l=p_{l}(\binom{r}{0}a^0,\ldots,\binom{r}{r}a^r)$, which is the denominator in \eqref{BP3}, does not seem to have a closed-form expression and makes the conjecture difficult to tackle (for example, when $3\leq l\leq 9$ and $a=1$, it is possible to show that this sum does not have closed form using techniques in \cite{petkovsek1996}. See \cite{calkin1998}).

This paper is organized as follows. In Section 2, we first show the log-concavity of the sequence \eqref{BP4} and also show that the sequence reaches its maximum value at $r=r_{m}=\big\lfloor\frac{m+2}{3}\big\rfloor$. In Section 3, we analyze the asymptotic behavior of the maximum value of this sequence $g_{m}(r_{m})$ as $m$ approaches infinity. In Section 4, we comment on Conjecture \ref{Conjecture} and the generalized sequence \eqref{BP3}.

\section{Log-concavity and peak location of the sequence}

Log-concavity of the sequence $\big\{g_{m}(r)\big\}_{0\leq r\leq m}$ can be verified using the following lemma, which was proved and used in \cite{byun2024maximum} to show the log-concavity of the sequence \eqref{BP1}.

\begin{lemma}\label{Lemma2.1}
    Let $\{x_{i}\}_{i\geq0}$ and $\{y_{i}\}_{i\geq0}$ be two sequences of nonnegative real numbers that are log-concave. For each nonnegative integer $k$, we set $z_{k}\coloneqq\sum_{i=0}^{k}x_{i}y_{i}$. Then, the sequence $\{z_{i}\}_{i\geq0}$ is log-concave.
\end{lemma}

\begin{proof}[Proof of Theorem \ref{MainThm} (a)]
It is well known (and not hard to check) that the sequence $\{\binom{m}{i}\}_{0\leq i\leq m}$ is log-concave. If we set $x_{i}=y_{i}=\binom{m}{i}$ in Lemma \ref{Lemma2.1}, we can conclude that the sequence $\{\sum_{i=0}^{r}\binom{m}{i}^2\}_{0\leq i\leq m}$ is log-concave.
It is straightforward to check that the sequence $\{\binom{2r}{r}^{-1}\}_{0\leq r\leq m}$ is also log-concave, using the definition of log-concavity: for any integer $r$ such that $1\leq r\leq m-1$,
\begin{equation*}
    \frac{\big(\binom{2r}{r}^{-1}\big)^2}{\binom{2r-2}{r-1}^{-1}\binom{2r+2}{r+1}^{-1}}=\frac{(r!)^{4}(2r-2)!(2r+2)!}{(2r!)^{2}((r-1)!)^2((r+1)!)^2}=\frac{r(2r+1)}{(r+1)(2r-1)}=\frac{2r^{2}+r}{2r^{2}+r-1}>1.
\end{equation*}
Since the product of two log-concave sequences is also log-concave, the sequence $\big\{g_{m}(r)\big\}_{0\leq r\leq m}=\{\binom{2r}{r}^{-1}\sum_{i=0}^{r}\binom{m}{i}^2\}_{0\leq r\leq m}$ is also log-concave. Since every log-concave sequence of nonnegative numbers is unimodal, the sequence $\big\{g_{m}(r)\big\}_{0\leq r\leq m}$ is also unimodal.
\end{proof}

To prove part (b) of Theorem~\ref{MainThm}, since the sequence is unimodal, it suffices to show the following two inequalities:
\begin{equation}\label{Inequality1:Prop3.1}
    g_{m}(r_{m}-1) < g_{m}(r_{m})
\end{equation}
and
\begin{equation}\label{Inequality2:Prop3.1}
    g_{m}(r_{m}) > g_{m}(r_{m}+1)
\end{equation}
for every integer $m\geq2$ and $r_{m}=\big\lfloor\frac{m+2}{3}\big\rfloor$.

We first discuss~\eqref{Inequality1:Prop3.1}. Note that
\begin{equation}\label{Equivalence1:Prop3.1}
    g_{m}(r_{m}-1) < g_{m}(r_{m}) \iff (3r_{m}-2)\sum_{i=0}^{r_{m}-1}\binom{m}{i}^2 < r_{m}\binom{m}{r_{m}}^2.
\end{equation}
The following lemma will be useful to verify the right-hand side of~\eqref{Equivalence1:Prop3.1}.
\begin{lemma}\label{Lemma3.2}
    For $m,j\in\mathbb{Z}^{+}$ such that $m\geq 2$ and $1\leq j\leq r_{m}=\big\lfloor\frac{m+2}{3}\big\rfloor$, we have
\begin{equation}\label{Inequality1:Lemma3.2}
    (3r_{m}-2)\sum_{i=0}^{j-1}\binom{m}{i}^2 < r_{m}\binom{m}{j}^2.
\end{equation}    
\end{lemma}
\begin{proof}
    One can directly check that inequality \eqref{Inequality1:Lemma3.2} holds for $m=2,\dots,6$. For a fixed integer $m\geq7$, we use induction on $j$ to prove \eqref{Inequality1:Lemma3.2}.  Note that in this case, since $m\geq7$, we have $r_{m}\geq \big\lfloor\frac{7+2}{3}\big\rfloor=3$.  When $j=1$, \eqref{Inequality1:Lemma3.2} becomes
    \begin{equation*}
        3r_{m}-2 < r_{m}m^2 \iff \frac{3r_{m}-2}{r_{m}} < m^2,
    \end{equation*}
    which is true because $\frac{3r_{m}-2}{r_{m}} < 3 < m^2$ for $m\geq 7$.
    Also, when $j=2$, \eqref{Inequality1:Lemma3.2} becomes
    \begin{equation*}
        (3r_{m}-2)(1+m^2) < r_{m}\bigg(\frac{m(m-1)}{2}\bigg)^2 \iff \frac{3r_{m}-2}{r_{m}} < \frac{m^4-2m^3+m^2}{4(m^2+1)}=\frac{m^2-2m}{4}+\frac{m}{2(m^2+1)},
    \end{equation*}
    which is also true because $\frac{3r_{m}-2}{r_{m}} < 3 < \frac{m^2-2m}{4}+\frac{m}{2(m^2+1)}$ for $m\geq 7$.

    Now we assume that \eqref{Inequality1:Lemma3.2} holds for all positive integers $j$ less than or equal to an integer $k$ such that $2\leq k < r_{m}$, and we show that \eqref{Inequality1:Lemma3.2} is true when $j=k+1$. By the induction hypothesis, setting $j=k-1$ in \eqref{Inequality1:Lemma3.2} gives us
    \begin{equation}\label{Inequality2:Lemma3.2}
        (3r_{m}-2)\sum_{i=0}^{k-2}\binom{m}{i}^2 < r_{m}\binom{m}{k-1}^2.
    \end{equation}
    To show that \eqref{Inequality1:Lemma3.2} also hold when $j=k+1$, we need to show:
    \begin{equation}\label{Inequality3:Lemma3.2}
        (3r_{m}-2)\sum_{i=0}^{k}\binom{m}{i}^2 < r_{m}\binom{m}{k+1}^2.
    \end{equation}
    By \eqref{Inequality2:Lemma3.2},
    \begin{equation*}
    \begin{aligned}
        (3r_{m}-2)\sum_{i=0}^{k}\binom{m}{i}^2 &= (3r_{m}-2)\sum_{i=0}^{k-2}\binom{m}{i}^2 + (3r_{m}-2)\bigg(\binom{m}{k-1}^2+\binom{m}{k}^2\bigg)\\
        &< r_{m}\binom{m}{k-1}^2 + (3r_{m}-2)\bigg(\binom{m}{k-1}^2+\binom{m}{k}^2\bigg)\\
        &=(4r_{m}-2)\binom{m}{k-1}^2 + (3r_{m}-2)\binom{m}{k}^2.
    \end{aligned}
    \end{equation*}
    Thus, it is enough to show that the inequality
    \begin{equation*}
        (4r_{m}-2)\binom{m}{k-1}^2 + (3r_{m}-2)\binom{m}{k}^2 \leq r_{m}\binom{m}{k+1}^2
    \end{equation*}
    holds for an integer $k$ such that $2\leq k<r_{m}$.
    If we divide both sides of the above inequality by its right-hand side, we get
    \begin{equation}\label{Inequality4:Lemma3.2}
        \frac{4r_{m}-2}{r_{m}}\cdot\frac{k^2(k+1)^2}{(m-k)^2(m-k+1)^2} + \frac{3r_{m}-2}{r_{m}}\cdot\frac{(k+1)^2}{(m-k)^2} \leq 1.
    \end{equation}
    One can view the left-hand side of \eqref{Inequality4:Lemma3.2} as a rational function in $k$. Note that this rational expression is an increasing function in $k$ when $k$ lies on the interval $[0,m)$, which contains $[2,r_{m})$. Thus, it suffices to show that the inequality is true when $k=r_{m}-1$. If we set $k=r_{m}-1$ in \eqref{Inequality4:Lemma3.2}, we obtain the following inequality:
    \begin{equation}\label{Inequality5:Lemma3.2}
        \frac{r_{m}(r_{m}-1)^2(4r_{m}-2)}{(m-r_{m}+1)^2(m-r_{m}+2)^2} + \frac{r_{m}(3r_{m}-2)}{(m-r_{m}+1)^2} \leq 1.
    \end{equation}
    We check this inequality by considering three cases, depending on the class of $m$ modulo $3$.

    \textbf{Case 1)} $m=3q-2$ for an integer $q\geq3$. In this case, $r_{m}=\big\lfloor\frac{m+2}{3}\big\rfloor=\big\lfloor\frac{3q}{3}\big\rfloor=q$. If we replace $m$ and $r_{m}$ by $3q-2$ and $q$, respectively, in \eqref{Inequality5:Lemma3.2}, then we get
    \begin{equation*}
        \frac{q(q-1)^2(4q-2)}{4q^2(2q-1)^2} + \frac{q(3q-2)}{(2q-1)^2} \leq 1,
    \end{equation*}
    which is true because we can rewrite the left-hand side of this inequality as follows:
    \begin{equation*}
        \frac{q(q-1)^2(4q-2)}{4q^2(2q-1)^2} + \frac{q(3q-2)}{(2q-1)^2} = \frac{16q^4-18q^3+8q^2-2q}{4q^2(2q-1)^2} = 1-\frac{(q-1)^2}{2q(2q-1)^2} < 1.
    \end{equation*}

    \textbf{Case 2)} $m=3q-1$ for an integer $q\geq3$. In this case, $r_{m}=\big\lfloor\frac{m+2}{3}\big\rfloor=\big\lfloor\frac{3q+1}{3}\big\rfloor=q$. If we replace $m$ and $r_{m}$ by $3q-1$ and $q$, respectively, in \eqref{Inequality5:Lemma3.2}, then we get
    \begin{equation*}
        \frac{q(q-1)^2(4q-2)}{4q^2(2q+1)^2} + \frac{q(3q-2)}{4q^2} \leq 1,
    \end{equation*}
    which is true because we can rewrite the left-hand side of this inequality as follows:
    \begin{equation*}
        \frac{q(q-1)^2(4q-2)}{4q^2(2q+1)^2} + \frac{q(3q-2)}{4q^2} = \frac{16q^{4}-6q^{3}+3q^{2}-4q}{4q^2(2q+1)^2}=1-\frac{22q^{2}+q+4}{4q(2q+1)^2}<1.
    \end{equation*}

    \textbf{Case 3)} $m=3q$ for an integer $q\geq3$. In this case, $r_{m}=\big\lfloor\frac{m+2}{3}\big\rfloor=\big\lfloor\frac{3q+2}{3}\big\rfloor=q$. If we replace $m$ and $r_{m}$ by $3q$ and $q$, respectively, in \eqref{Inequality5:Lemma3.2}, then we get
    \begin{equation*}
        \frac{q(q-1)^2(4q-2)}{(2q+1)^2(2q+2)^2} + \frac{q(3q-2)}{(2q+1)^2} \leq 1,
    \end{equation*}
    which is true because we can rewrite the left-hand side of this inequality as follows:
    \begin{equation*}
        \frac{q(q-1)^2(4q-2)}{(2q+1)^2(2q+2)^2} + \frac{q(3q-2)}{(2q+1)^2} = \frac{16q^{4}+6q^{3}+4q^{2}-10q}{(2q+1)^2(2q+2)^2}=1-\frac{(42q^{3}+48q^{2}+34q+4)}{(2q+1)^2(2q+2)^2}<1.
    \end{equation*}
Thus, \eqref{Inequality5:Lemma3.2} is verified for all integers $m\geq7$, and this completes the proof.
\end{proof}

For the inequality \eqref{Inequality2:Prop3.1}, note that 
\begin{equation}\label{Equivalence2:Prop3.1}
    g_{m}(r_{m}) > g_{m}(r_{m}+1) \iff \sum_{i=0}^{r_{m}}\binom{m}{i}^2 > \frac{r_{m}+1}{3r_{m}+1}\binom{m}{r_{m}+1}^2.
\end{equation}
We claim that it suffices to check the right-hand side of~\eqref{Equivalence2:Prop3.1} for all positive integers $m$ that are multiples of $3$. The following lemma will be used to confirm that claim.

\begin{lemma}\label{Lemma3.3}
    For $l,m,$ and $N\in\mathbb{Z}_{\geq0}$ such that $l+1 < m$,
    \begin{equation}\label{ImplicationinLemma3.3}
        \sum_{i=0}^{l}\binom{m}{i}^2 > N\binom{m}{l+1}^2 \implies \sum_{i=0}^{l}\binom{m-1}{i}^2 > N\binom{m-1}{l+1}^2.
    \end{equation}
\end{lemma}

\begin{proof}
\begin{equation*}
\begin{aligned}
    \sum_{i=0}^{l}\binom{m-1}{i}^2 = \sum_{i=0}^{l}\bigg(\frac{m-i}{m}\bigg)^2\cdot\binom{m}{i}^2 &> \sum_{i=0}^{l}\bigg(\frac{m-l-1}{m}\bigg)^2\cdot\binom{m}{i}^2\\ &= \bigg(\frac{m-l-1}{m}\bigg)^2\cdot\sum_{i=0}^{l}\binom{m}{i}^2\\
    &>\bigg(\frac{m-l-1}{m}\bigg)^2\cdot N\binom{m}{l+1}^2\\
    &=N\binom{m-1}{l+1}^2.
\end{aligned}
\end{equation*}
\end{proof}
We now explain why we only have to show the inequality on the right in \eqref{Equivalence2:Prop3.1} for positive integers $m$ that are multiples of $3$. Since $r_{m}=\big\lfloor\frac{m+2}{3}\big\rfloor$, for any positive integer $q$, the value of $r_{m}$ is equal to $q$ for $m=3q, 3q-1,$ and $3q-2$. If we set $m=3q$, $l=q$, and $N=\frac{q+1}{3q+1}$ in \eqref{ImplicationinLemma3.3}, then condition $l+1<m$ holds for all positive integers $q$, and we get

\begin{equation*}
        \sum_{i=0}^{q}\binom{3q}{i}^2 > \frac{q+1}{3q+1}\binom{3q}{q+1}^2 \implies \sum_{i=0}^{q}\binom{3q-1}{i}^2 > \frac{q+1}{3q+1}\binom{3q-1}{q+1}^2,
    \end{equation*}
which means that if the inequality on the right in \eqref{Equivalence2:Prop3.1} is true for $m=3q$, then it is also true for $m=3q-1$.  Similarly, if we set $m=3q-1$, $l=q$, and $N=\frac{q+1}{3q+1}$ in \eqref{ImplicationinLemma3.3}, then condition $l+1<m$ holds for all positive integers $q\geq2$, and we have

\begin{equation}
        \sum_{i=0}^{l}\binom{3q-1}{i}^2 > \frac{q+1}{3q+1}\binom{3q-1}{q+1}^2 \implies \sum_{i=0}^{l}\binom{3q-2}{i}^2 > \frac{q+1}{3q+1}\binom{3q-2}{q+1}^2,
    \end{equation}
which implies that if the inequality on the right side in \eqref{Equivalence2:Prop3.1} is true for $m=3q-1$, then it is also true for $m=3q-2$. Thus, our goal is to prove
\begin{equation}\label{Finalgoal}
    (3q+1)\sum_{i=0}^{q}\binom{3q}{i}^{2} > (q+1)\binom{3q}{q+1}^{2}
\end{equation}
for positive integers $q$. To prove this inequality, we define two families of polynomials $X_{i}(t)$ and $Y_{i}(t)$ for nonnegative integers $i$ as follows:
\begin{itemize}
    \item $X_{0}(t)\coloneqq t+1$ and $Y_{0}(t)\coloneqq 1$.
    \item $Y_{n+1}(t)\coloneqq (t+1-n)^2Y_{n}(t)=\prod_{i=1}^{n+1}(t+2-i)^2$ for $n \in \mathbb{Z}_{\geq 0}$.
    \item $X_{n+1}(t)\coloneqq (2t+n)^2X_{n}(t)-(3t+1)Y_{n+1}(t)$ for $n \in \mathbb{Z}_{\geq 0}$.
\end{itemize}
The polynomials $X_{n}(t)$ and $Y_{n}(t)$ for $n=0,\ldots,5$ are given in Table \ref{table1}. These polynomials are related to \eqref{Finalgoal} via the following lemma\footnote{Throughout this paper, the empty sum is understood as $0$. For example, when $k=q+1$, $\sum_{i=0}^{q-k}\binom{3q}{i}^{2}=\sum_{i=0}^{-1}\binom{3q}{i}^{2}=0$.}.

\begin{table}
\begin{tabular}
{ | m{0.3cm} | m{9.5cm}| m{5.5cm} | } 
  \hline
  $n$ & $X_{n}(t)$ & $Y_{n}(t)$ \\ 
  \hline
  $0$ & $t+1$ & $1$ \\ 
  \hline
  $1$ & $t^{3}-3t^{2}-5t-1$ & $t^{2}+2t+1$ \\
  \hline
  $2$ & $t^{5}-15t^{4}-36t^{3}-28t^{2}-9t-1$ & $t^{4}+2t^{3}+t^{2}$ \\ 
  \hline
  $3$ & $t^{7}-53t^{6}-254t^{5}-458t^{4}-407t^{3}-189t^{2}-44t-4$ & $t^{6}-2t^{4}+t^{2}$\\ 
  \hline
  $4$ & $t^{9}-189t^{8}-1645t^{7}-5383t^{6}-9397t^{5}-9743t^{4}-6115t^{3}-2249t^{2}-444t-36$ & $t^{8}-4t^{7}+2t^{6}+8t^{5}-7t^{4}-4t^{3}+4t^{2}$ \\ 
  \hline
  $5$ & $t^{11}-711t^{10}-9683t^{9}-50791t^{8}-149885t^{7}-275745t^{6}-330705t^{5}-262509t^{4}-135648t^{3}-43268t^{2}-7680t-576$ & $t^{10}-10t^{9}+35t^{8}-40t^{7}-37t^{6}+110t^{5}-35t^{4}-60t^{3}+36t^{2}$ \\ 
  \hline
\end{tabular}
    \caption{The polynomials $X_{n}(t)$ and $Y_{n}(t)$ for $n=0,\ldots,5$.}
    \label{table1}
\end{table}

\begin{lemma}\label{Lemma3.4}
    For $q\in\mathbb{Z^{+}}$ and $k\in\mathbb{Z}_{\geq0}$ such that $0\leq k\leq q+1$,
\begin{equation}\label{Equivalence:Lemma3.4}
    (3q+1)\sum_{i=0}^{q}\binom{3q}{i}^{2} > (q+1)\binom{3q}{q+1}^{2} \iff (3q+1)\sum_{i=0}^{q-k}\binom{3q}{i}^{2} > \frac{X_{k}(q)}{Y_{k}(q)}\binom{3q}{q+1-k}^{2}.
\end{equation}    
\end{lemma}
\begin{proof}
    We prove the lemma using induction on $k$. When $k=0$, it is trivially true since $\frac{X_{0}(q)}{Y_{0}(q)}=\frac{q+1}{1}=q+1$. Suppose \eqref{Equivalence:Lemma3.4} holds for some nonnegative integer $k$ such that $0\leq k<q+1$. If we subtract $(3q+1)\binom{3q}{q-k}^{2}$ from both sides of the inequality on the right side of \eqref{Equivalence:Lemma3.4}, then by the recursive definitions of two families of the polynomials $X_{n}$ and $Y_{n}$,
    \begin{equation*}
    \begin{aligned}
        (3q+1)\sum_{i=0}^{q-(k+1)}\binom{3q}{i}^{2} &> \frac{X_{k}(q)}{Y_{k}(q)}\binom{3q}{q+1-k}^{2}-(3q+1)\binom{3q}{q-k}^{2}\\
        &= \frac{X_{k}(q)}{Y_{k}(q)}\bigg(\frac{2q+k}{q+1-k}\bigg)^{2}\binom{3q}{q-k}^{2}-(3q+1)\binom{3q}{q-k}^{2}\\
        &= \frac{(2q+k)^{2}X_{k}(q)-(3q+1)(q+1-k)^{2}Y_{k}(q)}{(q+1-k)^{2}Y_{k}(q)}\binom{3q}{q-k}^{2}\\
        &= \frac{X_{k+1}(q)}{Y_{k+1}(q)}\binom{3q}{q+1-(k+1)}^{2}.
    \end{aligned}
    \end{equation*}
    Thus, the claim holds for $k+1$, which completes the proof.
\end{proof}
Note that if we set $k=q+1$ on the right side of \eqref{Equivalence:Lemma3.4}, we get the following inequality:
\begin{equation*}
    0>\frac{X_{q+1}(q)}{Y_{q+1}(q)}.
\end{equation*}
Thus, if we show $X_{q+1}(q)<0$ and $Y_{q+1}(q)>0$, then \eqref{Finalgoal} is verified (by Lemma \ref{Lemma3.4}) and \eqref{Inequality2:Prop3.1} follows (by Lemma \ref{Lemma3.3}). Since the latter inequality $Y_{q+1}(q)=\prod_{i=1}^{q+1}(q+2-i)^2>0$ is trivial, our task is to show that $X_{q+1}(q)<0$. This requires an in-depth analysis of the coefficients of the polynomials $X_{n}(t)$ and $Y_{n}(t)$. By their definitions, it is not hard to check that $X_{n}(t)$ and $Y_{n}(t)$ are monic polynomials of degrees $(2n+1)$ and $2n$, respectively. We then write two polynomials as follows\footnote{$b_{n,2n}$ is always $1$, so it seems meaningless to denote $1$ by $b_{n,2n}$. However, since using this notation can simplify some recurrence relations that we will describe later, we sometimes use $b_{n,2n}$ to denote the leading coefficient of $Y_{n}(t)$.}:
\begin{equation*}
X_{n}(t)=t^{2n+1}-\sum_{i=0}^{2n}a_{n,i}t^{i} \text{ 
and } Y_{n}(t)=t^{2n}+\sum_{i=0}^{2n-1}b_{n,i}t^{i}=\sum_{i=0}^{2n}b_{n,i}t^{i}.
\end{equation*}

Using the recursive definitions of the polynomials $X_{n}(t)$ and $Y_{n}(t)$, we can find recurrence relations that the coefficients $a_{n,i}$ and $b_{n,i}$ satisfy. For example, from $Y_{n+1}(t)\coloneqq (t+1-n)^2Y_{n}(t)$ for $n \in \mathbb{Z}_{\geq 0}$, we have
\begin{equation*}
    t^{2n+2}+\sum_{i=0}^{2n+1}b_{n+1,i}t^{i}=(t+1-n)^{2}\Bigg(t^{2n}+\sum_{i=0}^{2n-1}b_{n,i}t^{i}\Bigg).
\end{equation*}
If we compare the coefficients of both sides of the equation above, we have\footnote{Recall that $b_{n,2n}=1$. Also, if we set $b_{n,-2}=b_{n,-1}\coloneqq0$, then \eqref{bnjrecurrence3} and \eqref{bnjrecurrence4} can be regarded as \eqref{bnjrecurrence2} with $j=1$ and $j=0$, respectively.}
\begin{equation}\label{bnjrecurrence1}
    b_{n+1,2n+1}=b_{n,2n-1}-(2n-2),
\end{equation}
\begin{equation}\label{bnjrecurrence2}
    b_{n+1,j}=b_{n,j-2}-(2n-2)b_{n,j-1}+(n-1)^2b_{n,j} \text{ for } 2\leq j\leq 2n,
\end{equation}
\begin{equation}\label{bnjrecurrence3}
    b_{n+1,1}=-(2n-2)b_{n,0}+(n-1)^{2}b_{n,1},
\end{equation}
and
\begin{equation}\label{bnjrecurrence4}
    b_{n+1,0}=(n-1)^{2}b_{n,0}.
\end{equation}
Using these recurrence relations, one can find values of $b_{n,j}$ for $j=0,1,2,2n-3,2n-2,$ and $2n-1$.
\begin{lemma}\label{Lemma3.5}
    For an integer $n\geq1$, we have
    \begin{equation*}
        b_{n,2n-1}=3n-n^{2} \text{ and } b_{n,2n-2}=\frac{n}{6}(3n^{3}-20n^{2}+36n-13).
    \end{equation*}
Also, for $n\geq2$, we have $b_{n,2}=[(n-2)!]^2$, $b_{n,1}=b_{n,0}=0$, and
    \begin{equation*}
        b_{n,2n-3}=-\frac{1}{6}n(n-1)(n-2)(n-3)(n^{2}-5n+2).
    \end{equation*}
\end{lemma}
\begin{proof}
    First, $b_{n,0}=0$ for $n\geq2$ can be obtained by solving \eqref{bnjrecurrence4} with the initial condition $b_{2,0}=0$.
    
    Combining $b_{n,0}=0$ for $n\geq2$ and \eqref{bnjrecurrence3}, we get a new recurrence
    \begin{equation*}
        b_{n+1,1}=(n-1)^{2}b_{n,1}
    \end{equation*}
    for $n\geq2$. Using this recurrence and the initial condition $b_{2,1}=0$, one gets $b_{n,1}=0$ for $n\geq2$.

    Similarly, combining $b_{n,1}=b_{n,0}=0$ for $n\geq2$ and \eqref{bnjrecurrence2} with $j=2$, we obtain a recurrence
    \begin{equation*}
        b_{n+1,2}=(n-1)^{2}b_{n,2}
    \end{equation*}
    for $n\geq2$. Solving this recurrence with the initial condition $b_{2,2}=1$ gives $b_{n,2}=[(n-2)!]^2$ for $n\geq2$.
    
    Next, $b_{n,2n-1}=3n-n^{2}$ can be obtained by solving \eqref{bnjrecurrence1} with the initial condition $b_{1,1}=2$.
    
    To find $b_{n,2n-2}$, we set $j=2n$ in \eqref{bnjrecurrence2} and use $b_{n,2n-1}=3n-n^{2}$ and $b_{n,2n}=1$. Then we have
    \begin{equation*}
        b_{n+1,2n}=b_{n,2n-2}+2n^{3}-7n^{2}+4n+1
    \end{equation*}
    and can get $b_{n,2n-2}=\frac{n}{6}(3n^{3}-20n^{2}+36n-13)$ by solving this recurrence relation with the initial condition $b_{1,0}=1.$
    
    Lastly, if we set $j=2n-1$ in \eqref{bnjrecurrence2} and use $b_{n,2n-2}=\frac{n}{6}(3n^{3}-20n^{2}+36n-13)$ and $b_{n,2n-1}=3n-n^{2}$, then we have
    \begin{equation*}
        b_{n+1,2n-1}=b_{n,2n-3}-\frac{1}{3}n(n-1)(n-2)(3n^{2}-11n+2).
    \end{equation*}
    Solving the recurrence relation with the initial condition $b_{2,1}=0$ gives the desired formula for $b_{n,2n-3}$ for $n\geq2$.
\end{proof}

Now, we use the recursive definition of $X_{n}(t)$ to obtain the recurrence relations that $a_{n,i}$ satisfies. If we compare the coefficients of the equation $X_{n+1}(t)=(2t+n)^{2}X_{n}(t)-(3t+1)Y_{n+1}(t)$, we have
\begin{equation}\label{anjreucrrence1}
    a_{n+1,2n+2}=(-4n+4a_{n,2n})+(1+3b_{n+1,2n+1}),
\end{equation}
\begin{equation}\label{anjreucrrence2}
    a_{n+1,2n+1}=(-n^{2}+4na_{n,2n}+4a_{n,2n-1})+(b_{n+1,2n+1}+3b_{n+1,2n}),
\end{equation}
\begin{equation}\label{anjreucrrence3}
    a_{n+1,j}=(n^{2}a_{n,j}+4na_{n,j-1}+4a_{n,j-2})+(b_{n+1,j}+3b_{n+1,j-1}) \text{ for } 2\leq j\leq 2n,
\end{equation}
\begin{equation}\label{anjreucrrence4}
    a_{n+1,1}=(n^{2}a_{n,1}+4na_{n,0})+(b_{n+1,1}+3b_{n+1,0}),
\end{equation}
and
\begin{equation}\label{anjreucrrence5}
    a_{n+1,0}=n^{2}a_{n,0}+b_{n+1,0}.
\end{equation}
Using these recurrence relations and Lemma \ref{Lemma3.5}, one can find the explicit values of $a_{n,2n}$, $a_{n,2n-1}$, $a_{n,1},$ and $a_{n,0}$. Furthermore, one can also find a lower bound of $a_{n,j}$ for integers $n$ and $j$ such that $n\geq2$ and $0\leq j\leq 2n$. These will be presented in Lemmas \ref{Lemma3.6} and \ref{Lemma3.7}, respectively.
\begin{lemma}\label{Lemma3.6}
    For an integer $n\geq1$, we have
    \begin{equation*}
    \begin{aligned}
        &a_{n,2n}=n^{2}+n+\frac{1}{3}(2^{2n+1}-5), a_{n,2n-1}=\frac{(3n^{2}-3n+29)4^{n}}{9}-\frac{9n^{4}+12n^{3}+24n^{2}+39n+58}{18},\\
        &a_{n,1}=\Bigg(4\sum_{i=1}^{n-1}\frac{1}{i}+5\Bigg)[(n-1)!]^2, \text{ and } a_{n,0}=[(n-1)!]^2.
    \end{aligned}
    \end{equation*}
\end{lemma}
\begin{proof}
First, if we combine \eqref{anjreucrrence5} with $b_{n,0}=0$ for $n\geq2$, then we get a new recurrence
\begin{equation*}
    a_{n+1,0}=n^{2}a_{n,0} \text{ for } n\geq1.
\end{equation*}
Then $a_{n,0}=[(n-1)!]^{2}$ follows from the above recurrence and the initial condition $a_{1,0}=1$.

If we combine \eqref{anjreucrrence4} with $b_{n,0}=b_{n,1}=0$ for $n\geq2$ and $a_{n,0}=[(n-1)!]^{2}$ for $n\geq1$, then we have a new recurrence
\begin{equation*}
    a_{n+1,1}=n^{2}a_{n,1}+4n[(n-1)!]^{2} \text{ for } n\geq1,
\end{equation*}
which is equivalent to
\begin{equation*}
    \frac{a_{n+1,1}}{[n!]^{2}}=\frac{a_{n,1}}{[(n-1)!]^{2}}+\frac{4}{n} \text{ for } n\geq1.
\end{equation*}
If we solve this recurrence relation with the initial condition $a_{1,1}=5$, we get the formula of $a_{n,1}$ given in the statement of the lemma.

Using $b_{n,2n-1}=(3n-n^{2})$ for $n\geq1$ and \eqref{anjreucrrence1}, we get a recurrence relation
\begin{equation*}
    a_{n+1,2n+2}=4a_{n,2n}-3n^{2}-n+7 \text{ for } n\geq1
\end{equation*}
and if we solve this recurrence with the initial condition $a_{1,2}=3$, we get $a_{n,2n}=n^{2}+n+\frac{1}{3}(2^{2n+1}-5)$.

To solve $a_{n,2n-1}$, we put formulas for $a_{n,2n}$, $b_{n+1,2n+1}$, and $b_{n+1,2n}$ in \eqref{anjreucrrence2}, then we get
\begin{equation*}
    a_{n+1,2n+1}=4a_{n,2n-1}+\frac{n}{6}(9n^{3}-6n+4^{n+2}-1)+5 \text{ for } n\geq1.
\end{equation*}
If we solve this recurrence with the initial condition $a_{1,1}=5$, then the solution is precisely the expression for $a_{n,2n-1}$ given in the statement of the lemma.
\end{proof}

As mentioned right before Lemma \ref{Lemma3.6}, we now find a lower bound of $a_{n,j}$ for integers $n$ and $j$ such that $n\geq2$ and $0\leq j\leq2n$. It is stated as Lemma \ref{Lemma3.7} below and is the key to proving \eqref{Inequality2:Prop3.1}.

\begin{lemma}\label{Lemma3.7}
For integers $n$ and $j$ such that $n\geq 2$ and $0\leq j\leq 2n$, we have
    \begin{equation}\label{Inequality1inLemma3.7}
        a_{n,j}\geq|b_{n,j-2}|+(2n)|b_{n,j-1}|+n^{2}|b_{n,j}|,
    \end{equation}
where we set $b_{n,-2}=b_{n,-1}\coloneqq 0$. As a consequence, we have $a_{n,j}\geq0$.
\end{lemma}
\begin{proof}
We prove this lemma using induction. For $n=2$, $j\in\{0,\ldots,4\}$ and we can easily check that \eqref{Inequality1inLemma3.7} holds in these cases. Indeed, for $(n,j)=(2,0),\ldots,(2,4)$, we have

\begin{equation*}
    a_{2,0}\geq|b_{2,-2}|+(4)|b_{2,-1}|+(2)^{2}|b_{2,0}| \iff 1\geq |0|+4|0|+(2)^{2}|0|,
\end{equation*}
\begin{equation*}
    a_{2,1}\geq|b_{2,-1}|+(4)|b_{2,0}|+(2)^{2}|b_{2,1}| \iff 9\geq |0|+4|0|+(2)^{2}|0|,
\end{equation*}
\begin{equation*}
    a_{2,2}\geq|b_{2,0}|+(4)|b_{2,1}|+(2)^{2}|b_{2,2}| \iff 28\geq |0|+4|0|+(2)^{2}|1|,
\end{equation*}
\begin{equation*}
    a_{2,3}\geq|b_{2,1}|+(4)|b_{2,2}|+(2)^{2}|b_{2,3}| \iff 36\geq |0|+4|1|+(2)^{2}|2|,
\end{equation*}
and
\begin{equation*}
    a_{2,4}\geq|b_{2,2}|+(4)|b_{2,3}|+(2)^{2}|b_{2,4}| \iff 15\geq |1|+4|2|+(2)^{2}|1|.
\end{equation*}
This confirms the statement for $n=2$. We now assume that the claim holds for integers $n$ and $j$ such that $n\geq2$ and $0\leq j\leq 2n$. Under this assumption, we show that the claim still holds for $n+1$ and $j$ such that $0\leq j\leq2n+2$. More precisely, we need to show
\begin{equation}\label{Inequality2inLemma3.7}
        a_{n+1,j}\geq|b_{n+1,j-2}|+(2n+2)|b_{n+1,j-1}|+(n+1)^{2}|b_{n+1,j}|
\end{equation}
for $j=0,\ldots,2n+2$.

Since $a_{n+1,0},a_{n+1,1}>0$ (by Lemma \ref{Lemma3.6}) and $b_{n+1,-2}=b_{n+1,-1}=b_{n+1,0}=b_{n+1,1}=0$ (by Lemma \ref{Lemma3.5}), \eqref{Inequality2inLemma3.7} is true for $j=0$ and $1$. We now consider the cases when $j=2n+1$ and $2n+2$. When $j=2n+2$, \eqref{Inequality2inLemma3.7} becomes the following inequality.
\begin{equation}\label{Inequality3inLemma3.7}
        a_{n+1,2n+2}\geq|b_{n+1,2n}|+(2n+2)|b_{n+1,2n+1}|+(n+1)^{2}|b_{n+1,2n+2}|.
\end{equation}
Using Table \ref{table1}, one can easily check that the above inequality holds for $n=2$ (When $n=2$, \eqref{Inequality3inLemma3.7} becomes $53\geq11$), so we only have to prove it for $n\geq3$. Using explicit expressions for $a_{n+1,2n+2}$ (in Lemma \ref{Lemma3.6}), $b_{n+1,2n}$ and $b_{n+1,2n+1}$ (in Lemma \ref{Lemma3.5}), and $b_{n+1,2n+2}=1$, we can rewrite this inequality as follows:
\begin{equation*}
\begin{aligned}
    (n+1)^{2}+(n+1)+\frac{1}{3}(2^{2n+3}-5)\geq &\Big|\frac{(n+1)}{6}(3(n+1)^{3}-20(n+1)^{2}+36(n+1)-13)\Big|\\
    &+(2n+2)|3(n+1)-(n+1)^{2}|+(n+1)^{2}.
\end{aligned}
\end{equation*}
One can check $3(n+1)^{3}-20(n+1)^{2}+36(n+1)-13>0$ and $3(n+1)-(n+1)^{2}<0$ for $n\geq3$. Thus, we need to verify
\begin{equation*}
\begin{aligned}
    (n+1)^{2}+(n+1)+\frac{1}{3}(2^{2n+3}-5)\geq &\frac{(n+1)}{6}(3(n+1)^{3}-20(n+1)^{2}+36(n+1)-13)\\
    &+(2n+2)((n+1)^{2}-3(n+1))+(n+1)^{2}
\end{aligned}
\end{equation*}
for $n\geq 3$. If we solve this inequality for $2^{2n+3}$, we gets
\begin{equation*}
    2^{2n+3}\geq\frac{3}{2}n^4+2n^3-3n^2-\frac{31}{2}n-7.
\end{equation*}
A routine derivative argument establishes that this inequality is true for real numbers $n\geq3$. More precisely, we consider a function
\begin{equation*}
    f(x)=2^{2x+3}-\bigg(\frac{3}{2}x^4+2x^3-3x^2-\frac{31}{2}x-7\bigg).
\end{equation*}
Then, it suffices to show that $f(x)>0$ for $x\geq3$. One can check that $f(3), f'(3),\ldots,f^{(4)}(3)>0$ and $f^{(5)}(x)=(2\ln2)^5\cdot2^{2x+3}>0$ for all $x\geq 3$, and thus we conclude that $f(x), f'(x),\ldots,f^{(4)}(x)>0$ for all $x\geq3$. This verifies that \eqref{Inequality3inLemma3.7} is true for $n\geq3$, and thus we can conclude that \eqref{Inequality3inLemma3.7} is true for $n\geq2$.

We can use almost the same approach to verify the case $j=2n+1$. In this case, \eqref{Inequality2inLemma3.7} becomes the following inequality.
\begin{equation}\label{Inequality4inLemma3.7}
        a_{n+1,2n+1}\geq|b_{n+1,2n-1}|+(2n+2)|b_{n+1,2n}|+(n+1)^{2}|b_{n+1,2n+1}|.
\end{equation}
The cases when $n=2$ and $3$ can be easily checked using Table \ref{table1} (when $n=2$ and $3$, \eqref{Inequality4inLemma3.7} becomes $254\geq12$ and $1645\geq88$, respectively). Again, using explicit expressions for $a_{n+1,2n+1}$ (in Lemma \ref{Lemma3.6}), $b_{n+1,2n-1}$, $b_{n+1,2n}$, and $b_{n+1,2n+1}$ (in Lemma \ref{Lemma3.5}), we can rewrite this inequality as follows:
\begin{equation*}
\begin{aligned}
    &\frac{(3(n+1)^{2}-3(n+1)+29)4^{n+1}}{9}-\frac{9(n+1)^{4}+12(n+1)^{3}+24(n+1)^{2}+39(n+1)+58}{18}\\
    \geq &\Big|-\frac{1}{6}(n+1)n(n-1)(n-2)((n+1)^{2}-5(n+1)+2)\Big|\\
    &+(2n+2)\Big|\frac{(n+1)}{6}(3(n+1)^{3}-20(n+1)^{2}+36(n+1)-13)\Big|+(n+1)^{2}|3(n+1)-(n+1)^{2}|.
\end{aligned}
\end{equation*}
One can check $(n+1)^{2}-5(n+1)+2>0$, $3(n+1)^{3}-20(n+1)^{2}+36(n+1)-13>0$, and $3(n+1)-(n+1)^{2}<0$ for $n\geq4$. Thus, we need to verify 
\begin{equation*}
\begin{aligned}
    &\frac{(3(n+1)^{2}-3(n+1)+29)4^{n+1}}{9}-\frac{9(n+1)^{4}+12(n+1)^{3}+24(n+1)^{2}+39(n+1)+58}{18}\\
    \geq &\frac{1}{6}(n+1)n(n-1)(n-2)((n+1)^{2}-5(n+1)+2)\\
    &+\frac{1}{3}(n+1)^2(3(n+1)^{3}-20(n+1)^{2}+36(n+1)-13)+(n+1)^{2}((n+1)^{2}-3(n+1))
\end{aligned}
\end{equation*}
for $n\geq 4$. Again, using a routine derivative argument, one can verify that this inequality is valid for real numbers $n\geq4$. If one solves this inequality for $4^{n+1}$, it becomes
\begin{equation*}
\begin{aligned}
    4^{n+1}
    \geq &\frac{3n^6+3n^5+6n^4+9n^3+78n^2+159n+142}{2(3n^2+3n+29)}\\
    =&\bigg(\frac{1}{2}n^4-\frac{23}{6}n^2+\frac{16}{3}n+\frac{805}{18}\bigg)-\frac{3768n+22067}{18(3n^2+3n+29)}.
\end{aligned}
\end{equation*}
Since $\frac{3768n+22067}{18(3n^2+3n+29)}>0$ for $n\geq4$, it suffices to show that 
\begin{equation*}
    4^{n+1}
    \geq \frac{1}{2}n^4-\frac{23}{6}n^2+\frac{16}{3}n+\frac{805}{18}
\end{equation*}
for $n\geq4$. This time, we consider a function
\begin{equation*}
    g(x)=4^{x+1}-\bigg(\frac{1}{2}x^4-\frac{23}{6}x^2+\frac{16}{3}x+\frac{805}{18}\bigg).
\end{equation*}
One can then check that $g(4), g'(4),\ldots,g^{(4)}(4)>0$ and $g^{(5)}(x)=(\ln4)^5\cdot4^{x+1}>0$ for all $x\geq 4$, and thus we conclude that $g(x), g'(x),\ldots,g^{(4)}(x)>0$ for all $x\geq4$.
This confirms \eqref{Inequality4inLemma3.7} for $n\geq4$ and thus for $n\geq 2$.

To complete the proof, we still need to show \eqref{Inequality2inLemma3.7} for integers $j$ between $2$ and $2n$. Note that for any integers $j$ such that $2\leq j\leq2n$, we have\footnote{First one is \eqref{anjreucrrence3}, second, third, and fourth ones are from the induction hypothesis, and the last three follow from \eqref{bnjrecurrence2}, \eqref{bnjrecurrence3}, and \eqref{bnjrecurrence4}. Note that with the conventions $b_{n,-2}\coloneqq0$ and $b_{n,-1}\coloneqq0$, one can view \eqref{bnjrecurrence3} and \eqref{bnjrecurrence4} as \eqref{bnjrecurrence2} with $j=1$ and $j=0$, respectively.}
\begin{itemize}
    \item $a_{n+1,j}=(n^{2}a_{n,j}+4na_{n,j-1}+4a_{n,j-2})+(b_{n+1,j}+3b_{n+1,j-1})$,
    \item $a_{n,j}\geq|b_{n,j-2}|+(2n)|b_{n,j-1}|+n^{2}|b_{n,j}|$,
    \item $a_{n,j-1}\geq|b_{n,j-3}|+(2n)|b_{n,j-2}|+n^{2}|b_{n,j-1}|$,
    \item $a_{n,j-2}\geq|b_{n,j-4}|+(2n)|b_{n,j-3}|+n^{2}|b_{n,j-2}|$,
    \item $b_{n+1,j}=b_{n,j-2}-(2n-2)b_{n,j-1}+(n-1)^2b_{n,j}$,
    \item $b_{n+1,j-1}=b_{n,j-3}-(2n-2)b_{n,j-2}+(n-1)^2b_{n,j-1}$, and
    \item $b_{n+1,j-2}=b_{n,j-4}-(2n-2)b_{n,j-3}+(n-1)^2b_{n,j-2}$.
\end{itemize}
If we combine these equations and inequalities, we get
\begin{align*}
    a_{n+1,j}
    & = (n^{2}a_{n,j}+4na_{n,j-1}+4a_{n,j-2})+(b_{n+1,j}+3b_{n+1,j-1})\\
    &\geq  n^{2}(|b_{n,j-2}|+(2n)|b_{n,j-1}|+n^{2}|b_{n,j}|)+4n(|b_{n,j-3}|+(2n)|b_{n,j-2}|+n^{2}|b_{n,j-1}|)\\
    &\quad+4(|b_{n,j-4}|+(2n)|b_{n,j-3}|+n^{2}|b_{n,j-2}|)+(b_{n,j-2}-(2n-2)b_{n,j-1}+(n-1)^2b_{n,j})\\
    &\quad+3(b_{n,j-3}-(2n-2)b_{n,j-2}+(n-1)^2b_{n,j-1})\\
    &\geq 4|b_{n,j-4}|+(4n+8n-3)|b_{n,j-3}|+(n^{2}+8n^{2}+4n^{2}-1-3(2n-2))|b_{n,j-2}|\\
    &\quad+(2n^{3}+4n^{3}-(2n-2)-3(n-1)^{2}))|b_{n,j-1}|+(n^{4}-(n-1)^{2})|b_{n,j}|\\
    &\geq |b_{n,j-4}|+(4n)|b_{n,j-3}|+(6n^{2}-2)|b_{n,j-2}|+(4n^{3}-4n)|b_{n,j-1}|+(n^{4}-2n^{2}+1)|b_{n,j}|\\
    & = |b_{n,j-4}|+\{(2n-2)+(2n+2)\}|b_{n,j-3}|+\{(n-1)^{2}+(2n+2)(2n-2)+(n+1)^{2}\}|b_{n,j-2}|\\
    &\quad+\{(2n+2)(n-1)^{2}+(2n-2)(n+1)^{2}\}|b_{n,j-1}|+\{(n+1)^{2}(n-1)^{2}\}|b_{n,j}|\\ 
    & = \Big[|b_{n,j-4}|+(2n-2)|b_{n,j-3}|+(n-1)^{2}|b_{n,j-2}|\Big]+(2n+2)\Big[|b_{n,j-3}|+(2n-2)|b_{n,j-2}|\\
    &\quad+(n-1)^{2}|b_{n,j-1}|\Big]+(n+1)^{2}\Big[|b_{n,j-2}|+(2n-2)|b_{n,j-1}|+(n-1)^{2}|b_{n,j}|\Big]\\
    &\geq |b_{n,j-4}-(2n-2)b_{n,j-3}+(n-1)^2b_{n,j-2}|+(2n+2)|b_{n,j-3}-(2n-2)b_{n,j-2}\\
    &\quad+(n-1)^2b_{n,j-1}|+(n+1)^{2}|b_{n,j-2}-(2n-2)b_{n,j-1}+(n-1)^2b_{n,j}|\\
    & =|b_{n+1,j-2}|+(2n+2)|b_{n+1,j-1}|+(n+1)^{2}|b_{n+1,j}|
\end{align*}
for any integers $2\leq j\leq 2n$, where
\begin{itemize}
    \item in the first inequality, we used the induction hypothesis and the recurrences for $b_{n+1,j-1}$ and $b_{n+1,j-2}$,
    \item in the second inequality, we used $x\geq-|x|$ for $x\in\mathbb{R}$,
    \item in the third inequality, we used several inequalities $4\geq1$, $4n+8n-3=12n-3\geq4n$, $n^2+8n^2+4n^2-1-3(2n-2)=13n^2-6n+5\geq 6n^2-2$, $2n^3+4n^3-(2n-2)-3(n-1)^2=6n^3-3n^2+4n-1\geq 4n^3-4n$, and $n^4-(n-1)^2=n^4-n^2+2n-1\geq n^4-2n^2+1$, which hold for $n\geq2$,
    \item in the fourth inequality, we used a triangular inequality ($|a|+|b|+|c|\geq|a-b+c|$).
\end{itemize}
Thus, we conclude that \eqref{Inequality1inLemma3.7} is true for $n+1$ and all $j$ between $0$ and $2n+2$. This completes the proof of the induction step and we conclude that \eqref{Inequality1inLemma3.7} holds for all integers $n$ and $j$ such that $n\geq 2$ and $0\leq j\leq 2n$. This completes the proof.
\end{proof}

We now prove the second part of Theorem \ref{MainThm}.
\begin{proof}[Proof of Theorem \ref{MainThm} (b)]
If we set $j=r_{m}$ in \eqref{Inequality1:Lemma3.2}, then we get the right side of \eqref{Equivalence1:Prop3.1} and it verifies $g_{m}(r_{m}-1) < g_{m}(r_{m})$ for all integers $m\geq2$. This proves \eqref{Inequality1:Prop3.1}.

To prove \eqref{Inequality2:Prop3.1}, as explained right after Lemma \ref{Lemma3.4}, it suffices to show
\begin{equation*}
    X_{q+1}(q)=q^{2q+3}-\sum_{i=0}^{2q+2}a_{q+1,i}q^{i}<0    
\end{equation*}
for positive integers $q$. Since $a_{q+1,2q+2}=(q+1)^{2}+(q+1)+\frac{1}{3}(2^{2q+3}-5)$ (by Lemma \ref{Lemma3.6}) is strictly greater than $q$ and $a_{q+1,j}\geq0$ for integers $0\leq j\leq 2q+1$ (by Lemma \ref{Lemma3.7}), we have
\begin{equation*}
    X_{q+1}(q)=q^{2q+3}-\sum_{i=0}^{2q+2}a_{q+1,i}q^{i}\leq q^{2q+3}-a_{q+1,2q+2}q^{2q+2}=q^{2q+2}(q-a_{q+1,2q+2})<0.
\end{equation*}
This completes the proof of \eqref{Inequality2:Prop3.1} and Theorem \ref{MainThm} (b).
\end{proof}

\section{Asymptotic behavior of the maximum value}

In this section, we analyze the asymptotic behavior of the maximum value $g_{m}(r_{m})$ of the sequence $\big\{g_{m}(r)\big\}_{0\leq r\leq m}=\{\binom{2r}{r}^{-1}\sum_{i=0}^{r}\binom{m}{i}^2\}_{0\leq r\leq m}$ and prove Theorem \ref{MainThm} (c). If we solve \eqref{Equivalence1:Prop3.1} and \eqref{Equivalence2:Prop3.1} for $\sum_{i=0}^{r_{m}}\binom{m}{i}^2$, we get
\begin{equation*}
    \frac{r_{m}+1}{3r_{m}+1}\binom{m}{r_{m}+1}^2 < \sum_{i=0}^{r_{m}}\binom{m}{i}^2 < \frac{4r_{m}-2}{3r_{m}-2}\binom{m}{r_{m}}^2.
\end{equation*}
Since $\binom{m}{r_{m}+1}=\frac{m-r_{m}}{r_{m}+1}\binom{m}{r_{m}}$, the above inequality implies
\begin{equation}\label{Inequality for Asymptotic}
    \frac{r_{m}+1}{3r_{m}+1}\bigg(\frac{m-r_{m}}{r_{m}+1}\bigg)^2\frac{1}{\binom{2r_{m}}{r_{m}}}\binom{m}{r_{m}}^2 < g_{m}(r_{m})=\frac{1}{\binom{2r_{m}}{r_{m}}}\sum_{i=0}^{r_{m}}\binom{m}{i}^2 < \frac{4r_{m}-2}{3r_{m}-2}\frac{1}{\binom{2r_{m}}{r_{m}}}\binom{m}{r_{m}}^2.
\end{equation}
Using \eqref{Inequality for Asymptotic}, we now prove Theorem \ref{MainThm} (c): to do that, we take the limit (as $m\rightarrow \infty)$ to the inequality \eqref{Inequality for Asymptotic} and analyze the limiting behaviors of the left and right sides of the inequality.

\begin{proof}[Proof of Theorem \ref{MainThm} (c)]
Using $r_{m}=\big\lfloor\frac{m+2}{3}\big\rfloor$, one can show that $\displaystyle\lim_{m\rightarrow\infty}\frac{r_{m}+1}{3r_{m}+1}\big(\frac{m-r_{m}}{r_{m}+1}\big)^2=\lim_{m\rightarrow\infty}\frac{4r_{m}-2}{3r_{m}-2}=\frac{4}{3}$. Thus, by taking limit in \eqref{Inequality for Asymptotic}, we conclude that as $m\rightarrow\infty$,\footnote{$f(m)\sim g(m)$ means $\lim_{m\rightarrow\infty}\frac{f(m)}{g(m)}=1$.}
\begin{equation}\label{Asymptotic}
    g_{m}(r_{m})\sim \frac{4}{3}\frac{1}{\binom{2r_{m}}{r_{m}}}\binom{m}{r_{m}}^2=\frac{4}{3}\frac{(m!)^2}{[(m-r_{m})!]^2(2r_{m})!}.
\end{equation}
An application of Stirling's formula\footnote{$n!\sim\sqrt{2\pi n}\big(\frac{n}{e}\big)^n$ as $n\rightarrow\infty$.} yields
\begin{equation*}
\begin{aligned}
    \frac{4}{3}\frac{(m!)^2}{[(m-r_{m})!]^2(2r_{m})!} &\sim \frac{4}{3}\frac{2\pi m\big(\frac{m}{e}\big)^{2m}}{2\pi (m-r_{m})\big(\frac{m-r_{m}}{e}\big)^{2(m-r_{m})}\sqrt{4\pi r_{m}}\big(\frac{2r_{m}}{e}\big)^{2r_{m}}}\\
    &=\frac{2\sqrt{2}m^{2m+1}}{3\sqrt{\pi}(m-r_{m})^{2(m-r_{m})+1}(2r_{m})^{2r_{m}+\frac{1}{2}}}\\
    &=\frac{3^{2m+\frac{1}{2}}}{2^{2m}\sqrt{\pi m}}\cdot\bigg(\frac{2m}{3(m-r_{m})}\bigg)^{2(m-r_{m})+1}\cdot\bigg(\frac{m}{3r_{m}}\bigg)^{2r_{m}+\frac{1}{2}}.
\end{aligned}
\end{equation*}
Thus, to prove Theorem \ref{MainThm} (c), it suffices to show that
\begin{equation}\label{Asymptotic2}
    \bigg(\frac{2m}{3(m-r_{m})}\bigg)^{2(m-r_{m})+1}\cdot\bigg(\frac{m}{3r_{m}}\bigg)^{2r_{m}+\frac{1}{2}}\sim1
\end{equation}
holds. We can verify it by considering three classes of $m$ modulo $3$. Let $m=3k-\epsilon$, where $k$ is a positive integer and $\epsilon\in\{0,1,2\}$. Then, $r_{m}=\big\lfloor\frac{m+2}{3}\big\rfloor=\big\lfloor\frac{(3k-\epsilon)+2}{3}\big\rfloor=k$ for every $\epsilon\in\{0,1,2\}$. We first consider the cases where $\epsilon=1$ or $\epsilon=2$. In these cases, if we set $m=3k-\epsilon$ and $r_{m}=k$ in the left side of \eqref{Asymptotic2}, then as\footnote{We use the fact that $\lim_{n\rightarrow \infty}(1-\frac{1}{n})^{n}=e^{-1}.$} $k\rightarrow\infty$,
\begin{equation*}
\begin{aligned}
    \bigg(\frac{2m}{3(m-r_{m})}\bigg)^{2(m-r_{m})+1}\cdot\bigg(\frac{m}{3r_{m}}\bigg)^{2r_{m}+\frac{1}{2}}
    =~&\bigg(\frac{6k-2\epsilon}{6k-3\epsilon}\bigg)^{4k-2\epsilon+1}\cdot\bigg(\frac{3k-\epsilon}{3k}\bigg)^{2k+\frac{1}{2}}\\
    \sim~&\bigg(\frac{6k-2\epsilon}{6k-3\epsilon}\bigg)^{4k}\cdot\bigg(\frac{3k-\epsilon}{3k}\bigg)^{2k}\\
    =~&\bigg(\frac{6k-2\epsilon}{6k}\bigg)^{4k}\cdot\bigg(\frac{6k-3\epsilon}{6k}\bigg)^{-4k}\cdot\bigg(\frac{3k-\epsilon}{3k}\bigg)^{2k}\\
    =~&\bigg(1-\frac{\epsilon}{3k}\bigg)^{4k}\cdot\bigg(1-\frac{\epsilon}{2k}\bigg)^{-4k}\cdot\bigg(1-\frac{\epsilon}{3k}\bigg)^{2k}\\
    \sim~&e^{-\frac{4}{3}\epsilon}\cdot e^{2\epsilon}\cdot e^{-\frac{2}{3}\epsilon}\\
    =~&1.
\end{aligned}
\end{equation*}
On the other hand, when $\epsilon=0$, if we set $m=3k$ and $r_{m}=k$ in the left side of \eqref{Asymptotic2}, then
\begin{equation*}
    \bigg(\frac{2m}{3(m-r_{m})}\bigg)^{2(m-r_{m})+1}\cdot\bigg(\frac{m}{3r_{m}}\bigg)^{2r_{m}+\frac{1}{2}}
    =\bigg(\frac{6k}{6k}\bigg)^{4k+1}\cdot\bigg(\frac{3k}{3k}\bigg)^{2k+\frac{1}{2}}=1.
\end{equation*}
Thus, \eqref{Asymptotic2} is verified and this completes the proof.
\end{proof}

\section{Some comments on the Conjecture \ref{Conjecture}}

\subsection*{\textbf{1.}} 

Unlike Theorem \ref{MainThm} (a) of the current paper and Theorem 1.1 (a) in \cite{byun2024maximum}, unimodality of the generalized sequence \eqref{BP3} (Conjecture \ref{Conjecture} (b)) is not fully implied by log-concavity of the same sequence (Conjecture \ref{Conjecture} (a)). Note that in these two theorems, the sequences are log-concave for every integer $m\geq 2$, so we can conclude the unimodality of the same sequences for every $m\geq 2$. In general, the sequence \eqref{BP3} is (conjecturally) log-concave for $m$ large enough, so the unimodality of the sequence for small values of $m$ needs to be proved with a different technique.

We confirmed unimodality and checked log-concavity of~\eqref{BP3} for integer values $1\leq a\leq10$, $1\leq l\leq20$, and $1\leq m\leq100$. For each pair $(a, l)$, Table~\ref{table2} gives the largest value $m \leq 100$ for which the sequence~\eqref{BP3} is not log-concave. We note that the entries are weakly increasing along each column. Although not shown in the table, we further observe that for fixed $a$ and $m$, if the sequence is not log-concave for $l=L\in\mathbb{Z}^+$, then the sequence is still not log-concave for all integers $l\geq L$. In fact, the latter observation implies the former one, but we do not know how to explain either of these observations.

\begin{table}
\begin{tabular}
{ | P{0.8cm} | P{0.5cm} | P{0.5cm} | P{0.5cm} | P{0.5cm} | P{0.5cm} | P{0.5cm} | P{0.5cm} | P{0.5cm} | P{0.5cm} | P{0.5cm} | } 
\hline
$l$ $\backslash$ $a$ & 1 & 2 & 3 & 4 & 5 & 6 & 7 & 8 & 9 & 10\\
\hline
1 & 0 & 0 & 0 & 0 & 0 & 0 & 0 & 0 & 0 & 0\\
\hline
2 & 0 & 0 & 0 & 0 & 0 & 0 & 0 & 0 & 0 & 0\\
\hline
3 & 0 & 0 & 0 & 0 & 0 & 0 & 0 & 0 & 0 & 0\\
\hline
4 & 0 & 0 & 0 & 0 & 0 & 0 & 0 & 0 & 0 & 0\\
\hline
5 & 0 & 0 & 0 & 0 & 0 & 0 & 0 & 0 & 0 & 0\\
\hline
6 & 5 & 0 & 0 & 0 & 0 & 0 & 0 & 0 & 0 & 0\\
\hline
7 & 5 & 0 & 0 & 0 & 0 & 0 & 0 & 0 & 0 & 0\\
\hline
8 & 9 & 0 & 0 & 0 & 0 & 0 & 0 & 0 & 0 & 0\\
\hline
9 & 12 & 7 & 10 & 0 & 0 & 0 & 0 & 0 & 0 & 0\\
\hline
10 & 13 & 7 & 10 & 0 & 0 & 0 & 0 & 0 & 0 & 0\\
\hline
11 & 16 & 10 & 15 & 13 & 15 & 0 & 0 & 0 & 0 & 0\\
\hline
12 & 19 & 12 & 15 & 18 & 21 & 24 & 28 & 0 & 0 & 0\\
\hline
13 & 20 & 14 & 20 & 19 & 21 & 25 & 28 & 32 & 35 & 39\\
\hline
14 & 24 & 16 & 20 & 24 & 28 & 25 & 29 & 32 & 36 & 39\\
\hline
15 & 27 & 17 & 25 & 25 & 28 & 33 & 37 & 42 & 46 & 40\\
\hline
16 & 30 & 20 & 25 & 25 & 29 & 33 & 38 & 42 & 47 & 51\\
\hline
17 & 31 & 21 & 30 & 31 & 35 & 41 & 46 & 42 & 47 & 52\\
\hline
18 & 35 & 25 & 30 & 31 & 36 & 41 & 46 & 52 & 57 & 63\\
\hline
19 & 39 & 25 & 35 & 37 & 42 & 41 & 47 & 52 & 58 & 63\\
\hline
20 & 42 & 29 & 35 & 37 & 43 & 49 & 55 & 62 & 68 & 64\\
  \hline
\end{tabular}
\vspace{3mm}
    \caption{The $(i+1,j+1)$ entry of this table is the largest value of $m$ no greater than $100$ such that the sequence \eqref{BP3} is not log-concave when $l=i$ and $a=j$.}
    \label{table2}
\end{table}

\subsection*{\textbf{2.}} It would be interesting to find an alternative proof of Theorem \ref{MainThm}, which could lead us to a proof of Conjecture \ref{Conjecture}. As explained in Section~\ref{sec:Intro}, our proof heavily relies on the fact that when $l=2$ and $a=1$, the denominator of the sequence \eqref{BP3}, $\sum_{i=0}^{r}\{\binom{r}{i}a^{i}\}^{l}$, is simplified to $\binom{2r}{r}$. For this reason, the proof of this paper cannot be adapted to prove Conjecture \ref{Conjecture} in general. 

Meanwhile, in their recent work \cite{glasby2024maximizing}, Glasby and Paseman provided a new proof of their previous result \cite{glasby2022maximum} and generalized it in a different direction than the authors. More specifically, they characterized the peak location of the sequence
\begin{equation*}
    \bigg\{\frac{1}{\omega^r}\sum_{i=0}^{r}\binom{m}{r}\bigg\}_{0\leq r\leq m}
\end{equation*}
for a nonnegative integer $m$ and a real number $\omega\geq1$, thus generalizing their previous result\footnote{The previous result corresponds to the case when $\omega=2$.}. Their proof was interesting, as they used a completely different approach from their previous work \cite{glasby2022maximum}. For the proof, they first expressed $Q\coloneqq(r+1)\binom{m}{r+1}/\sum_{i=0}^{r}\binom{m}{i}$ as a generalized continued fraction. Then, they obtained good approximations of $Q$ based on the generalized continued fraction and used them to find a value of $r$ that maximizes the sequence above. It would be interesting if one could use their approach to give a simpler proof of Theorem \ref{MainThm} or to prove other cases of Conjecture \ref{Conjecture}.

\bibliography{bibliography}{}
\bibliographystyle{abbrv}

\end{document}